\documentclass[final,1p,times]{elsarticle}

\usepackage{amsthm,amsmath,amssymb,mathrsfs,amsfonts,graphicx,graphics,latexsym,exscale,cmmib57,dsfont,amscd,ulem}
\usepackage{mathrsfs}
\usepackage{fancybox}
\usepackage{color,soul}
\usepackage[colorlinks=true]{hyperref}

\newcommand{\e}{\boldsymbol{o}}

\newcommand{\X}{\boldsymbol{X}}
\newcommand{\Y}{\boldsymbol{Y}}

\newcommand{\bea}{\begin{eqnarray}}
\newcommand{\eea}{\end{eqnarray}}
\newcommand{\bean}{\begin{eqnarray*}}
\newcommand{\eean}{\end{eqnarray*}}
\newtheorem*{FST}{Furstenberg Structure Theorem}

\newtheorem{Thm}{Theorem}[section]
\newtheorem{cor}[Thm]{Corollary}
\newtheorem{prop}[Thm]{Proposition}
\newtheorem{Lem}[Thm]{Lemma}

\theoremstyle{definition}
\newtheorem{defn}[Thm]{Definition}

\numberwithin{equation}{section}
\journal{xxx}

\begin{document}
\begin{frontmatter}

\title{An extension of Furstenberg's structure theorem for Noetherian modules and multiple recurrence theorems III}

\author{Xiongping Dai}
\ead{xpdai@nju.edu.cn}
\address{Department of Mathematics, Nanjing University, Nanjing 210093, People's Republic of China}

\begin{abstract}
Using a recent Furstenberg structure theorem, we obtain Multiple Recurrence Theorems relative to any locally compact second countable Noetherian module $G$ over a syndetic ring $R$, which generalizes Furstenberg's multiple recurrence theorem. In addition we study the multiple Birkhoff center and the pointwise multiple recurrence of a topological $G$-action on a compact metric space.
\end{abstract}

\end{frontmatter}


\section*{0. Introduction}\label{sec0}
As a subsequent work of \cite{Dai-pre, Dai-pre2}, this paper will continue to study the Furstenberg multiple recurrence of dynamical systems on a standard Borel probability space induced by any locally compact second countable Noetherian module over a syndetic ring.
\subsection*{0.1. Basic notions}
First of all, by a ``lcscN'' $R$-module $(G,\pmb{+})$ over a ``syndetic'' ring $(R,+,\cdot)$, we mean that $(G,\pmb{+})$ and $(R,+)$ both are locally compact second countable Hausdorff topological groups satisfying the following conditions (cf.~\cite{Dai-pre}):
\begin{itemize}
\item The multiplication operation of $R$ on $G$, $(t,g)\mapsto tg$, is continuous from $R\times G$ to $G$.
\item $G$ is a \textit{Noetherian} $R$-module: for every sequence $G_0\subseteq G_1\subseteq G_2\subseteq\dotsm$ of $R$-submodules of $G$, we have $G_n=G_{n+1}$ as $n$ sufficiently large.
\item $R$ is \textit{syndetic}: $\forall t\not=0$, $Rt$ is syndetic in the sense that one can find a compact subset $K$ of $R$ with $K+Rt=R$.
\end{itemize}
Moreover, we shall say that
\begin{itemize}
\item an $R$-module $G$ is \textit{irreducible} if $tg\not=\e$ whenever $R\ni t\not=0$ and $G\ni g\not=\e$, where $\e$ is the zero element of $G$ and $0$ is the zero element of $R$.
\end{itemize}

Clearly, $(\mathbb{Z}^n,+)$ over $\mathbb{Z}$, $(\mathbb{Q}^n,+)$ over $\mathbb{Q}$ and $(\mathbb{R}^d,+)$ over $\mathbb{R}$ all are irreducible lcscN modules over syndetic rings. Moreover, the $p$-adic integer (syndetic) ring $\mathbb{Z}_p$ and the $p$-adic number (syndetic) field $\mathbb{Q}_p$ are irreducible lcscN as module over themselves. See \cite{Lan} for more examples.

\subsection*{0.2. Measure-theoretic multiple recurrence theorems}
Let $(G,\pmb{+})$ be a topological $R$-module and let $(X,\mathscr{X},\mu)$ be a probability space. We will consider a measure-preserving $G$-action dynamical system:
\begin{gather*}
T\colon G\times X\rightarrow X\quad \textrm{or write}\quad G\curvearrowright_TX
\end{gather*}
where $T_g\colon X\rightarrow X$ is $\mu$-preserving for each $g\in G$ and the $G$-action map $T\colon (g,x)\mapsto T_g(x)$ is jointly measurable. For our convenience, we will write
\begin{gather*}
T_{tg}(x)=T_g^t(x)=g^tx\ \textit{ and so }\ T_{g\pmb{+}h}^t=g^t\circ h^t\quad \forall t\in R\textit{ and }g,h\in G
\end{gather*}
if no confusion. Then given any $g\not=\e$,
\begin{gather*}
T_g\colon R\times X\rightarrow X\quad  \textit{by }\ (t,x)\mapsto T_g^t(x)=g^tx
\end{gather*}
defines a new $\mu$-preserving Borel $R$-system.

Harry Furstenberg in 1977~\cite{F77} extraordinarily extended the classical Poincar\'{e} recurrence theorem to the \textit{multiple recurrence} as follows:
\begin{itemize}
\item For any cyclic measure-preserving system
$(X,\mathscr{X},\mu,T)$ and any integer $l\ge2$, the transformations $T,T^2,\dotsc,T^l$ have a common power satisfying $\mu(A\cap T^{-n}A\cap\dotsm\cap T^{-ln}A)>0$ for any set of positive $\mu$-measure.
\end{itemize}
There he also showed that this result implies Szemer\'{e}di's theorem asserting that any set of integers of positive upper density contains arbitrarily long arithmetic progressions~\cite{Sz}. In 1978~\cite{FW}, he and B.~Weiss proved a topological analogue\,---\,the multiple Birkhoff recurrence theorem:
\begin{itemize}
\item If $T$ is a homeomorphism of a compact metric space $X$, then for any $\varepsilon>0$ and any integer $l=1,2,\dotsc$, there is a point $x\in X$ and a common power $n$ of $T,T^2,\dotsc,T^l$ such that $d(x,T^nx)<\varepsilon, \dotsc, d(x,T^{ln}x)<\varepsilon$.
\end{itemize}
This weaker result in turn implies van der Waerden's theorem on arithmetic progressions for partitions/colorings of the integers. In fact this topological result is true for any $l$ commuting continuous transformations (cf.~\cite[Theorem~2.6]{Fur}).
In 1978~\cite{FK}, he and Y.~Katznelson showed that the measure-theoretic model of the multiple Birkhoff recurrence theorem is also true for arbitrary commuting transformations (cf.~\cite[Theorem~A]{FK} and \cite[Theorems~7.13, 7.14 and 7.15]{Fur}). A subsequent corollary is the multidimensional extension of Szemer\'{e}di's theorem on arithmetic progressions (cf.~\cite[Theorem~B]{FK} and \cite[Theorem~7.16]{Fur}).

Starting with Furstenberg's ergodic-theoretic methods introduced in his landmark papers \cite{F63,F77}, there have already been many important generalizations of the multiple recurrence theorem and Szemer\'{e}di's theorem; see, e.g., \cite{FK85, BL96, Lei, BL03, BHK, FHK, Fra, HK, Zie, Tao, Kra, Aus, CFH, Pot, FHK13, BT, KLOY} and references therein.

By $|\cdot|$ we mean a Haar measure on the lcsc ring $(R,+,\cdot)$. Recall that a sequence of nonull compact subsets $F_n$ of $R$ such that
\begin{equation*}
\lim_{n\to\infty}\frac{|(r+F_n)\vartriangle F_n|}{|F_n|}=0\quad\forall r\in R,
\end{equation*}
is called a \textit{weak F{\o}lner sequence} in $(R,+)$.\footnote{Here and in the future, unlike the case in literature, for a weak F{\o}lner sequence we neither need to require
\begin{equation*}
\lim_{n\to\infty}\frac{|(K+F_n)\vartriangle F_n|}{|F_n|}=0\quad\forall K\subset R\textit{ compact}
\end{equation*}
nor other regularities such as Tempelman condition~\cite{Tem} or Shulman condition (cf.~\cite{EW}).} Particularly, a weak F{\o}lner sequence $\{F_n\}_1^\infty$ in $(R,+)$ is said to be \textbf{\textit{asymptotic}} if for each $m\in\mathbb{N}$ there exists a constant $c_m>0$ such that
\begin{equation*}
mF_n\subseteq F_{mn}\quad\textit{and}\quad |F_{mn}|\le c_m|F_n| \quad \forall n\ge1.\eqno{(AF)}
\end{equation*}
For example, $F_n=\{0,1,2,\dotsc,n-1\}, n=1,2,\dotsc$ is an asymptotic F{\o}lner sequence in $(\mathbb{Z},+)$; $F_n=[0,n], n=1,2,\dotsc$ is an asymptotic F{\o}lner sequence in $(\mathbb{R},+)$ with the Euclidean topology.

In different flavor, this paper is mainly to prove the following measure-theoretic multiple recurrence theorem respecting to an irreducible lcscN module such as $\mathbb{Q}^n, \mathbb{Z}_p, \mathbb{Q}_p^n, \mathbb{R}^n$ beyond the integer group $\mathbb{Z}$.

\begin{Thm}\label{thm0.1}
Let $G$ be an irreducible lcscN $R$-module over a syndetic ring $(R,+,\cdot)$. Then every $G\curvearrowright_T(X,\mathscr{X},\mu)$ is an \textit{Sz-system} over asymptotic F{\o}lner sequences; i.e., for any
$A\in\mathscr{X}$ with $\mu(A)>0$ and any $g_1,\dotsc,g_l\in G, l\ge2$,
\begin{gather*}
\liminf_{n\to+\infty}\frac{1}{|F_n|}\int_{F_n}\int_Xg_1^t1_A\dotsm g_l^t1_A\,d\mu dt>0,
\end{gather*}
over any asymptotic F{\o}lner sequence $\{F_n\}_{1}^\infty$ in $(R,+)$.
\end{Thm}

The $\liminf\limits_{n\to+\infty}$ is actually a limit in the classical case that $(R,+,\cdot)=(\mathbb{Z},+,\cdot)$ and $G=\mathbb{Z}^n$; see \cite{HK, Zie, Tao, Aus}. However, this question is open in the present context.

We can obtain from Theorem~\ref{thm0.1} the continuous-time version of Furstenberg's Multiple Recurrence Theorem~\cite{F77, FK, Fur} as follows:

\begin{cor}\label{cor0.2}
Let $\varphi\colon\mathbb{R}^d\times(X,\mathscr{X},\mu)\rightarrow(X,\mathscr{X},\mu)$ be a Borel flow over a standard Borel probability space $(X,\mathscr{X},\mu)$. Then for $A\in\mathscr{X}$ with $\mu(A)>0$,
\begin{gather*}
\lim_{T\to+\infty}\frac{1}{T}\int_0^T\mu\left(A\cap g_1^{-t}A\cap\dotsm\cap g_l^{-t}A\right)dt>0
\end{gather*}
for any nonzero elements $g_1,\dotsc,g_l\in \mathbb{R}^d$ and $l\ge1$.
\end{cor}

Our condition $(AF)$ is similar to Calderon's original doubling condition~\cite{Cal} that is formulated for an increasing family of compact symmetric neighborhoods of the zero $0$ of $R$. We note that this condition is often incompatible with the weak F{\o}lner property, as noted in \cite{GE}. In view of this reason, we now restate our theorem as follows:

\begin{Thm}\label{thm0.3}
Let $G$ be any lcscN $R$-module over a syndetic ring $(R,+,\cdot)$. Then every measure-preserving $G$-system $G\curvearrowright_T(X,\mathscr{X},\mu)$ satisfies the following multiple recurrence property:
\begin{itemize}
\item For any
$A\in\mathscr{X}$ with $\mu(A)>0$ and any finite set $F=\{g_1,\dotsc,g_l\}\subseteq G$, there exists some integer $M=M(A,F)\ge1$ such that for any weak F{\o}lner sequence $\{F_n\}_{1}^\infty$ in $(R,+)$,
\begin{gather*}
\limsup_{n\to+\infty}\frac{1}{|F_n|}\int_{F_n}\int_Xg_1^{mt}1_A\dotsm g_l^{mt}1_A\,d\mu dt>0,
\end{gather*}
for some integer $m$ with $1\le m\le M$.
\end{itemize}
\end{Thm}
For simplicity, we shall call the property described in Theorem~\ref{thm0.3} the \textbf{\textit{weak \textit{Sz}-property}}. Following Furstenberg's correspondence principle, this weak \textit{Sz}-property is already enough to derive Szemer\'{e}di's theorem.

\subparagraph{Outline of the proof of Theorems~\ref{thm0.1} and \ref{thm0.3}}~\label{}
First we note that because $\{1,r_1t,\dotsc,r_lt\}$ as a family of polynomials, for distinct $r_1,\dotsc,r_l\in\mathbb{R}$, is not $\mathbb{R}$-independent for $l\ge2$, even the special case of $d=1$ in Corollary~\ref{cor0.2} above is not included in \cite[Theorem~1.2]{Pot} by using nilflows.

On the other hand, since an irreducible lcscN ring $(R,+,\cdot)$ does not need to be the classic field $\mathbb{R}$ and so there is no the discrete-time expression: given any $\delta\not=0$, for any $t\in R$ we have $t=n_t\delta+r_t$ with $n_t\in\mathbb{Z}$ and $0\le r_t<\delta$, hence the discretization methods developed recently by Bergelson et al. \cite{BLM} does not work for Theorem~\ref{thm0.1} here.

To prove our Multiple Recurrence Theorems in the probabilistic settings (Theorems~\ref{thm0.1} and \ref{thm0.3}), our main tool is the following structure theorem.

\begin{FST}[\cite{Dai-pre}]
Let $G$ be an lcscN $R$-module over a syndetic ring $R$ and $\X=(X,\mathscr{X},\mu,G)$ a nontrivial standard Borel $G$-system. Then there exists an ordinal $\eta$ and a system of factors $\left\{\pi_\xi\colon\X\rightarrow\X_\xi\right\}_{\xi\le\eta}$
such that
\begin{enumerate}
\item[$(a)$] $\X_0$ is the one-point $G$-system and $\X_\eta=\X$ ($\mu$-$\mathrm{mod}$ $0$).
\item[$(b)$] If $0\le\theta<\xi\le\eta$, then there is a factor $G$-map $\pi_{\xi,\theta}\colon\X_\xi\rightarrow\X_\theta$ with $\pi_\theta=\pi_{\xi,\theta}\circ\pi_\xi$.
\item[$(c)$] For each ordinal $\xi$ with $0\le\xi<\eta$, $\pi_{\xi+1,\xi}\colon\X_{\xi+1}\rightarrow\X_\xi$ is a nontrivial ``primitive'' extension.
\item[$(d)$] If $\xi$ is a limit ordinal $\le\eta$, then $\X_\xi=\underleftarrow{\lim}_{\theta<\xi}\X_\theta$.
\end{enumerate}
Moreover, the intermediate factors are of the form
\begin{gather*}
\X_\xi=(X,\mathscr{X}_\xi,\mu,G),\quad \pi_\xi=\textit{Id}_X\quad \textit{and}\quad \pi_{\xi+1,\xi}=\textit{Id}_X\quad (0<\xi<\eta).
\end{gather*}
Here we refer to
\begin{gather*}
\X\rightarrow\X_\eta\rightarrow\dotsm\rightarrow\X_{\xi+1}\rightarrow\X_\xi\rightarrow\dotsm\rightarrow\X_1\rightarrow\X_0
\end{gather*}
a ``Furstenberg factors chain'' of $\X$.
\end{FST}

Since $\X_0$ is an \textit{Sz}-system (and $\X_1$ is a \textit{Kh}-system \cite{Dai-pre2}), then as in the classical case, based on this Structure Theorem we mainly need to build three ladders to lift the \textit{Sz}-property of $\X_1$:

\begin{enumerate}
\item[(1).] For any limit factor $\X_\xi$ as in $(d)$, if each $\X_\theta, \theta<\xi$, is an (resp.~a weak) \textit{Sz}-system, $\X_\xi$ is also an (resp.~a weak) \textit{Sz}-system (Corollary~\ref{cor1.4} and Lemma~\ref{lem1.5} in $\S\ref{sec2}$).

\item[(2).] If an intermediate factor $\X_{\xi+1}=(X,\mathscr{X}_{\xi+1},\mu,G)$ of $\X$ is a totally relatively weak-mixing extension of an (resp.~a weak) \textit{Sz}-system $\X_\xi$, then $\X_{\xi+1}$ is an (resp.~a weak) \textit{Sz}-system; see Corollary~\ref{cor2.3} in $\S\ref{sec2}$.

\item[(3).] Consider any primitive link $\pi_{\xi+1,\xi}\colon\X_{\xi+1}\rightarrow\X_\xi$ for any order $\xi\ge1$, where $\X_\xi$ is an (resp.~a weak) \textit{Sz}-system. See Propositions~\ref{prop3.5} and \ref{prop3.6} in $\S\ref{sec3}$.
\end{enumerate}
Then we can easily prove our Multiple Recurrence Theorems by combining the above three ladders; see $\S\ref{sec4}$.
\subsection*{0.3. Topological multiple recurrence}
Next we will present some simple applications of Theorems~\ref{thm0.1} and \ref{thm0.3}.
Clearly Theorem~\ref{thm0.1} implies the following Topological Multiple Recurrence Theorem, which is an extension and strengthening of Furstenberg and Weiss~\cite[Theorem~1.5]{FW}:

\begin{prop}\label{prop0.4}
Let $T\colon G\times X\rightarrow X$ be a topological dynamical system on a compact metric space $X$, where $G$ is an irreducible lcscN module over a syndetic ring $R$. If $G\curvearrowright_TX$ is a \textsl{weak E-system} (i.e. $X$ is just the support of some $T$-invariant Borel probability measure of $X$ but $(X,T)$ does not need to be topologically transitive), then for any $g_1,\dotsc,g_l\in G$ and any nonempty open subset $U$ of $X$,
\begin{equation*}
N_{g_1,\dotsc,g_l}(U)=\left\{t\in R\,|\, U\cap g_1^{-t}U\cap\dotsm\cap g_l^{-t}U\not=\varnothing\right\}
\end{equation*}
is of positive lower density over any asymptotic F{\o}lner sequence in $(R,+)$.
\end{prop}

For the classical case where $G=\mathbb{Z}^d$, regarded as an lcscN $\mathbb{Z}$-module, acts minimally on $X$ and only requiring $N_{g_1,\dotsc,g_l}(U)\cap\mathbb{N}\not=\varnothing$, see \cite{FW, Fur} by using Bowen's lemma and homogeneous sets, \cite{BFHK} also~\cite[Theorem~1.56]{Gla} by using Ellis enveloping semigroup theory, and \cite{BPT} for topological proof of the topological multiple recurrence theorem.

Our further applications need the following notation.

\begin{defn}\label{def0.5}
A measurable subset $S$ of an lcsc ring $(R,+,\cdot)$ is called a \textit{F{\o}lner $\infty$-set} in $R$ if for any weak F{\o}lner sequence $\{F_n\}_1^\infty$ in $(R,+)$, there is some integer $m\ge1$ such that
the set $\{t\in R; mt\in S\}$ is of positive upper density over $\{F_n\}_1^\infty$.
\end{defn}
Then Theorem~\ref{thm0.3} implies the following

\begin{prop}\label{prop0.6}
Let $T\colon G\times X\rightarrow X$ be a topological dynamical system on a compact metric space $X$, where $G$ is an lcscN module over a syndetic ring $R$. If $G\curvearrowright_TX$ is a \textsl{weak E-system}, then for any sample elements $g_1,\dotsc,g_l\in G$ and any nonempty open subset $U$ of $X$,
$N_{g_1,\dotsc,g_l}(U)$ is a F{\o}lner $\infty$-set in $R$.
\end{prop}
Based on Def.~\ref{def0.5} we next introduce a concept---multiply nonwandering motion---for any topological dynamical system $G\curvearrowright_TX$ on a compact metric space $X$, which is the weakest multiple recurrence and which generalizes and strengthens the classical single nonwandering motion~\cite{NS}.

\begin{defn}\label{def0.7}
We say that a point $p\in X$ is \textit{multiply nonwandering} for $G\curvearrowright_TX$ if for any neighborhood $U$ of $p$, any sample elements $g_1, \dotsc,g_l\in G$, the set of multiple return times
\begin{equation*}
N_{g_1,\dotsc,g_l}(U)=\left\{t\in R\setminus\{0\}\colon U\cap g_1^{-t}U\cap\dotsm\cap g_l^{-t}U\not=\varnothing\right\},
\end{equation*}
is of F{\o}lner $\infty$-set in $R$. By $\Omega_F(G\curvearrowright_TX)$ it means the set of all the multiply nonwandering points of $G\curvearrowright_TX$. Here the subscript $_F$ is for Furstenberg because of his multiple recurrence.
\end{defn}

Requiring only $N_{g_1,\dotsc,g_l}(U)\not=\varnothing$ in place of our F{\o}lner $\infty$-set, this notion in fact was first introduced by Balcar et al. in 1987~\cite{BKW} in the more general situation of semigroup actions.

Although a multiply nonwandering point might have no multiple recurrence itself, yet multiple recurrence (even periodic motion) occurs near $p$.
By definitions, it is easy to check the following basic fact:

\begin{Lem}\label{lem0.8}
Let $G$ be any lcscN $R$-module over a syndetic ring $(R,+,\cdot)$. $\Omega_F(G\curvearrowright_TX)$ is an $T$-invariant closed subset of $X$ for any topological dynamical system $G\curvearrowright_TX$.
\end{Lem}

As a simple consequence of Theorem~\ref{thm0.3}, we can easily obtain the following result in the topological settings by considering the density points of ergodic probability measures.

\begin{prop}\label{prop0.9}
Let $G$ be any lcscN $R$-module over a syndetic ring $R$. Then for any topological dynamical system $G\curvearrowright_TX$ on a compact metric space $X$, $\Omega_F(G\curvearrowright_TX)$ is of full probability.
\end{prop}

\begin{proof}
Given any ergodic measure $\mu$ of $G\curvearrowright_TX$, let $\textrm{supp}(\mu)$ be the support of $\mu$. Then $\textrm{supp}(\mu)$ is of $\mu$-measure $1$ and for any $x\in\textrm{supp}(\mu)$ and any neighborhood $U$ of $x$ we have $\mu(U)>0$. Thus by Theorem~\ref{thm0.3}, it follows that $N_{g_1,\dotsc,g_l}(U)$ is a F{\o}lner $\infty$-set in $R$ for any finite set of sample elements $g_1,\dotsc,g_l\in G$. This concludes the proof of Proposition~\ref{prop0.9}.
\end{proof}

Because there always exists an ergodic Borel probability of $G\curvearrowright_TX$ for $G$ is amenable and $X$ is a compact metric space, $\Omega_F(G\curvearrowright_TX)$ is a nonempty subset of $X$.

In what follows we assume $G\curvearrowright_TX$ is a topological dynamical system on a compact metric space. Next let us note that the restriction of $T$ to $\Omega_F(G\curvearrowright_TX)$ is itself a topological $G$-space and further it has its own nonempty multiply nonwandering point set $\Omega_F(G\curvearrowright_T(\Omega_F(G\curvearrowright_TX)))$. We set
\begin{gather*}
\Omega_1(T)=\Omega_F(G\curvearrowright_TX),\ \Omega_2(T)=\Omega_F(G\curvearrowright_T\Omega_1(T)),\ \dotsc,\ \Omega_{\xi+1}(T)=\Omega_F(G\curvearrowright_T\Omega_\xi(T)),\ \dotsc
\intertext{and if $\xi$ is a limit ordinal, we define}\Omega_\xi(T)=\bigcap_{\theta<\xi}\Omega_\theta(T)\quad (\not=\varnothing, T\textrm{-invariant and closed}).
\end{gather*}
So by transfinite induction we can get a chain (possibly transfinite):
$$
\Omega_1(T)\supseteq \Omega_2(T)\supseteq\dotsm\supseteq \Omega_n(T)\supseteq \Omega_{n+1}(T)\supseteq\dotsm\supseteq \Omega_\omega(T)\subseteq \Omega_{\omega+1}(T)\supseteq\dotsm.
$$
Since $X$ is a compact Hausdorff space, then from the Cantor-Baire theorem it follows that there exists an ordinal $\gamma\ge1$ such that
\begin{gather*}
\Omega_\gamma(T)=\Omega_F(G\curvearrowright_T\Omega_\gamma(T)).
\end{gather*}

\begin{defn}\label{def0.10}
As in the classical case (cf., e.g.,~\cite[Chap.~V.5]{NS}), such $\Omega_\gamma(T)$ is referred to as the \textit{set of multiple center motions} or simply the \textit{multiple Birkhoff center} of order $\gamma$.
\end{defn}

Obviously, the multiple Birkhoff center $\Omega_\gamma(T)$ is an $T$-invariant, compact, and full probability set by Proposition~\ref{prop0.9} and its order $\gamma$ is not greater than a transfinite number of class I\!I.

Using the multiple Birkhoff center $\Omega_\gamma(T)$ as our tool, we shall show that there exists  a ``large'' set of multiply recurrent points in any topological $G$-space in $\S\ref{sec5.2}$. We note that the multiple Birkhoff center that is defined in a weaker manner of lacking the F{\o}lner $\infty$-set has been studied recently by Kwietniak et al~\cite{KLOY} using different approaches (cf.~e.g. \cite[$\S1.8$ and $\S3.1$]{Fur}) for $\mathbb{Z}$-action topological dynamical systems.

The remainder of this paper will be organized as follows.

\tableofcontents
\subsection*{Acknowledgments}
Finally, the author is deeply grateful to Professor Hillel Furstenberg for many helpful suggestions, comments, and carefully checking the details of the original manuscript.

\section{Inverse limits of \textit{Sz}-systems and weak \textit{Sz}-systems}\label{sec1}
In this section, let $(G,\pmb{+})$ be an lcscN $R$-module over a syndetic ring $(R,+,\cdot)$. Let $(X,\mathscr{X},\mu)$ be any $G$-space with the $G$-action map $T$ and write $\X=G\curvearrowright_T(X,\mathscr{X},\mu)$ and simply write $T_g^t=g^t$ for any $t\in R$ and $g\in G$.

First, following Hillel Furstenberg~\cite{Fur} we introduce the following concept.

\begin{defn}\label{def1.1}
We say that $\X$ is an \textit{Sz-system} over all asymptotic F{\o}lner sequences in $(R,+)$, where \textit{Sz} is for Szemer\'{e}di, provided that for any integer $l\ge2$,
\begin{itemize}
\item whenever $A\in\mathscr{X}$ with $\mu(A)>0$ and $g_1,\dotsc,g_l\in G$, then
\begin{gather}\label{eq1.1}
\liminf_{n\to+\infty}\frac{1}{|F_n|}\int_{F_n}\mu\left(g_1^{-t}A\cap\dotsm\cap g_l^{-t}A\right)dt>0
\end{gather}
over any asymptotic F{\o}lner sequence $\{F_n\}_1^\infty$ in $(R,+)$;
\end{itemize}
or equivalently,
\begin{itemize}
\item whenever $f\in L^\infty(X,\mathscr{X},\mu)$ with $f\ge0$ \textit{a.e.} and $\int_Xfd\mu>0$ and $g_1,\dotsc,g_l\in G$, then
\begin{gather}\label{eq1.2}
\liminf_{n\to+\infty}\frac{1}{|F_n|}\int_{F_n}\int_Xg_1^tf\dotsm g_l^tfd\mu dt>0
\end{gather}
over any asymptotic F{\o}lner sequence $\{F_n\}_1^\infty$ in $(R,+)$.
\end{itemize}
\end{defn}

We note that if $G$ is an $\mathbb{Z}$-module and $F_n=\{1,2,\dotsc,n\}$, then this definition reduces to the classical case of Furstenberg~\cite[Definition~7.1]{Fur}.

We will need the following simple equivalence later on.\footnote{There is a similar characterization to the property (\ref{eq1.2}) by an almost same argument with $\int_Xg_1^tf\dotsm g_l^tfd\mu$ in place of $\int_Xg_1^t1_A\dotsm g_l^t1_Ad\mu$.}

\begin{Lem}\label{lem1.2}
Let $\{F_n\}_1^\infty$ be any asymptotic F{\o}lner sequence in $(R,+)$.
Then an $A\in\mathscr{X}$ with $\mu(A)>0$ has the property $(\ref{eq1.1})$ over~$\{F_n\}_1^\infty$ if and only if there exist $\delta>0, \varepsilon>0$, and a Borel subset $\mathcal{H}$ of $R$ such that
\begin{gather*}
\mathrm{D}_*(\mathcal{H}):=\liminf_{n\to+\infty}\frac{|\mathcal{H}\cap F_n|}{|F_n|}>\delta\quad \textrm{and}\quad\int_Xg_1^t1_A(x)\dotsm g_l^t1_A(x)d\mu(x)>\varepsilon\ \forall t\in\mathcal{H}.
\end{gather*}
Here $\delta,\varepsilon$ depend upon $\{F_n\}_1^\infty$.
\end{Lem}

\begin{proof}
Sufficiency. Since
\begin{eqnarray*}
\lefteqn{\liminf_{n\to+\infty}\frac{1}{|F_n|}\int_{F_n}\int_Xg_1^{t}1_A\dotsm g_l^{t}1_Ad\mu dt}\\
& &\ge\liminf_{n\to+\infty}\frac{|\mathcal{H}\cap F_n|}{|F_n|}\frac{1}{|F_n\cap \mathcal{H}|}\int_{F_n\cap\mathcal{H}}\int_Xg_1^{t}1_A\dotsm g_l^{t}1_Ad\mu dt\\
& &\ge\delta\varepsilon,
\end{eqnarray*}
the sufficiency holds.

Necessity. Take $\varepsilon>0$ so that
\begin{gather*}
\liminf_{n\to+\infty}\frac{1}{|F_n|}\int_{F_n}\int_Xg_1^{t}1_A(x)\dotsm g_l^{t}1_A(x)d\mu(x)dt>3\varepsilon.
\end{gather*}
Set
$$
\mathcal{H}=\left\{t\in R\,\bigg{|}\,\int_Xg_1^{t}1_A(x)\dotsm g_l^{t}1_A(x)d\mu(x)>\varepsilon\right\}\quad \textrm{and}\quad \mathcal{H}^c=R\setminus \mathcal{H}.
$$
Clearly $\mathcal{H}$ is a Borel subset of $R$. To prove the necessity, we only need to prove that $\mathrm{D}_*(\mathcal{H})>0$ over $\{F_n\}_1^\infty$. Indeed, otherwise,
\begin{equation*}\begin{split}
3\varepsilon&<\liminf_{n\to+\infty}\frac{1}{|F_n|}\int_{F_n}\int_Xg_1^{t}1_A(x)\dotsm g_l^{t}1_A(x)d\mu(x)dt\\
&=\liminf_{n\to+\infty}\frac{|\mathcal{H}\cap F_n|}{|F_n|}\frac{1}{|\mathcal{H}\cap F_n|}\int_{F_n\cap \mathcal{H}}\int_Xg_1^{t}1_A(x)\dotsm g_l^{t}1_A(x)d\mu(x)dt\\
&{}\quad+\limsup_{n\to+\infty}\frac{1}{|F_n|}\int_{F_n\cap \mathcal{H}^c}\int_Xg_1^{t}1_A(x)\dotsm g_l^{t}1_A(x)d\mu(x)dt\\
&\le\limsup_{n\to+\infty}\frac{1}{|F_n|}\cdot\varepsilon|F_n\cap \mathcal{H}^c|\\
&\le\varepsilon
\end{split}\end{equation*}
which is a contradiction.

The proof of Lemma~\ref{lem1.2} is thus completed.
\end{proof}

The following is useful for us to prove the multiple recurrence theorem (Theorem~\ref{thm0.1}) by transfinite induction.

\begin{prop}\label{prop1.3}
Let $\{\mathscr{X}_\theta;\theta\in\Theta\}$ be a totally ordered family of $\sigma$-subalgebras of $\mathscr{X}$; that is to say, for any $\theta_1,\theta_2\in\Theta$ either $\mathscr{X}_{\theta_1}\subset\mathscr{X}_{\theta_2}$ or $\mathscr{X}_{\theta_2}\subset\mathscr{X}_{\theta_1}$. If each $G\curvearrowright_T(X,\mathscr{X}_\theta,\mu)$ is an \textit{Sz}-system and set $\mathscr{Y}=\sigma\big{(}\bigcup_\theta\mathscr{X}_\theta\big{)}$, then
$G\curvearrowright_T(X,\mathscr{Y},\mu)$ is an \textit{Sz}-system.
\end{prop}

\begin{proof}
Here our argument is different with and more concise than the proof of \cite[Proposition~7.1]{Fur}.\footnote{Inasmuch as the original proof of \cite[Proposition~7.1]{Fur} for $R=\mathbb{Z}$ (also see \cite[Proposition~7.26]{EW} for $G=\mathbb{Z}$) involves the disintegration $\{\mu_x;x\in X\}$ of $\mu$ given $\mathscr{X}_\theta$, whence there $(X,\mathscr{X},\mu)$ had to be a standard Borel space to obtain a disintegration of $\mu\colon\mu=\int_X\mu_xd\mu$ over $\mathscr{X}_\theta$ in \cite{FK, Fur, EW}.}
Let $A\in\mathscr{Y}$ with $\mu(A)>0$ and let $g_1,\dotsc,g_l\in G$. Since $\bigcup_\theta\mathscr{X}_\theta$ is an algebra generating $\mathscr{Y}$, by induction we can find $A_n\in\bigcup_\theta\mathscr{X}_\theta$ and $\epsilon_n\downarrow0$ so that
\begin{gather*}
\mu(A\vartriangle A_n)<\left(2^{-l}\epsilon_n\right)^2 \quad \textrm{and} \quad A_n\subseteq A_{n+1},\quad n=1,2,\dotsc.
\end{gather*}
Then $\|1_A-1_{A_n}\|_2<\epsilon_n2^{-l}$ in $\mathfrak{L}^2(X,\mathscr{Y},\mu)$ for each $n$.

Given any asymptotic F{\o}lner sequence $\{F_n\}_1^\infty$ in $(R,+)$ and $i_0\ge1$, by Lemma~\ref{lem1.2} it follows that for some $\varepsilon>0$ and $\delta>0$, the set
\begin{gather*}
\mathcal{H}_{i_0}=\left\{t\in R\,\bigg{|}\,\int_Xg_1^t1_{A_{i_0}}\dotsm g_l^t1_{A_{i_0}}d\mu>\varepsilon\right\}
\end{gather*}
has lower density $\ge\delta$ over $\{F_n\}_1^\infty$. Since $\{A_n\}$ is monotonically increasing,
\begin{gather*}
\mathcal{H}_i=\left\{t\in R\,\bigg{|}\,\int_Xg_1^{t}1_{A_{i}}\dotsm g_l^{t}1_{A_{i}}d\mu>\varepsilon\right\}\supseteq\mathcal{H}_{i_0}\quad\forall i\ge i_0.
\end{gather*}
Now by
\begin{gather*}
1_A(x)=1_{A_i}(x)+\psi_i(x)\quad \textrm{where}\ \|\psi_i\|_\infty\le1\textrm{ and }\|\psi_i\|_2<2^{-l}\epsilon_i
\end{gather*}
we can obtain that
\begin{equation*}
\begin{split}
\int_Xg_1^t1_A(x)\dotsm g_l^t1_A(x)d\mu(x)&=\int_Xg_1^t(1_{A_i}+\psi_i)\dotsm g_l^t(1_{A_i}+\psi_i)d\mu\\
&\ge\int_Xg_1^t1_{A_i}\dotsm g_l^t1_{A_i}d\mu-\epsilon_i\\
&>\varepsilon-\epsilon_i\quad \forall t\in\mathcal{H}_i,\ i\ge i_0.
\end{split}
\end{equation*}
Thus, by Lemma~\ref{lem1.2} once again, it follows that $G\curvearrowright_T(X,\mathscr{Y},\mu)$ is an \textit{Sz}-system.

This proves Proposition~\ref{prop1.3}.
\end{proof}

This result immediately leads to the following important fact, which is one of our three ladders for proving Theorem~\ref{thm0.1}.

\begin{cor}\label{cor1.4}
Let $\{\X_\xi; \xi\le\eta\}$ be a Furstenberg factors chain of a nontrivial standard Borel $G$-system $\X=G\curvearrowright_T(X,\mathscr{X},\mu)$. Assume that the ordinal $\xi$ is a limit ordinal and that $\X_\theta$ is an \textit{Sz}-system for each $\theta<\xi$. Then
$\X_\xi$ is an \textit{Sz}-system.
\end{cor}

\begin{proof}
According to Structure Theorem, $\{\mathscr{X}_\theta; \theta<\xi\}$ is totally ordered. Then the statement follows from the foregoing Proposition~\ref{prop1.3}.
\end{proof}

By an argument similar to that of Corollary~\ref{cor1.4}, we can easily obtain the following

\begin{Lem}\label{lem1.5}
Let $\{\X_\xi; \xi\le\eta\}$ be a Furstenberg factors chain of a nontrivial standard Borel $G$-system $\X=G\curvearrowright_T(X,\mathscr{X},\mu)$. Assume that the ordinal $\xi$ is a limit ordinal and that $\X_\theta$ is a weak \textit{Sz}-system for each $\theta<\xi$. Then
$\X_\xi$ is a weak \textit{Sz}-system.
\end{Lem}

\section{\textit{Sz}-property of totally relatively weak-mixing extensions}\label{sec2}
Let $(G,\pmb{+})$ be an lcsc $R$-module with the zero element $\e$ and with the continuous scalar multiplication $(t,g)\mapsto tg$ of $R\times G$ to $G$.
Let there be any given a short factors series:
\begin{gather*}
\X=G\curvearrowright_T(X,\mathscr{X},\mu)\xrightarrow[]{\textit{Id}_X}\X^\prime=G\curvearrowright_T(X,\mathscr{X}^\prime,\mu)\xrightarrow[]{\pi}\Y=G\curvearrowright_S(Y,\mathscr{Y},\nu),
\end{gather*}
where $(X,\mathscr{X},\mu)$ is a standard Borel $G$-space so we can decompose $\mu=\int_Y\mu_yd\nu(y)$ over $(\Y,\pi)$.

By $\mathrm{D}\textit{-}\lim$ we denote the limit in density (cf.~\cite{Fur,Dai-pre2}).
The following lemma is \cite[Proposition~3.7]{Dai-pre2}.

\begin{Lem}[\cite{Dai-pre2}]\label{lem2.1}
Let $\pi\colon\X^\prime\rightarrow\Y$ be totally relatively weak-mixing for $G$ and let $g_1,\dotsc,g_l\in G$ with $g_i\not=\e$ and $g_i\not=g_j$ for $1\le i\not=j\le l$. If
$f_1,\dotsc,f_l\in\mathfrak{L}^\infty(X,\mathscr{X}^\prime,\mu)$, then in the weak topology of $\mathfrak{L}^2(X,\mathscr{X}^\prime,\mu)$, it holds that
\begin{gather*}
\mathrm{D}\textit{-}\lim_{t\in R} \left\{\prod_{i=1}^lT_{g_i}^tf_i-\prod_{i=1}^lT_{g_i}^tE_\mu(f_i|\pi^{-1}[\mathscr{Y}])\right\}=0
\end{gather*}
over any weak F{\o}lner sequence in $(R,+)$.
\end{Lem}

We can restate this lemma as the following more convenient version for our later arguments:

\begin{Lem}\label{lem2.2}
Let $\pi\colon\X^\prime\rightarrow\Y$ be totally relatively weak-mixing for $G$ and let $g_1,\dotsc,g_l\in G$ with $g_i\not=g_j$ for $1\le i\not=j\le l$. If
$\psi_1,\dotsc,\psi_l\in\mathfrak{L}^\infty(X,\mathscr{X}^\prime,\mu)$, then in the weak topology of $\mathfrak{L}^2(X,\mathscr{X}^\prime,\mu)$, it holds that
\begin{gather*}
\mathrm{D}\textit{-}\lim_{t\in R} \left\{\prod_{i=1}^lT_{g_i}^t\psi_i-\prod_{i=1}^lT_{g_i}^tE_\mu(\psi_i|\pi^{-1}[\mathscr{Y}])\right\}=0
\end{gather*}
over any weak F{\o}lner sequence in $(R,+)$.
\end{Lem}

\begin{proof}
If $g_i\not=\e$ for all $1\le i\le l$, then this is just Lemma~\ref{lem2.1}. If not, replace all the $g_i$ by some $g_0\pmb{+}g_i$ where ${g_0}\not=-g_i$ for any $1\le i\le l$.
Since $\mu$ is $g_0^t$-invariant for each $t\in R$, this implies the desired statement.
\end{proof}

This lemma makes it evident that a totally relatively weak-mixing extension of an \textit{Sz}-factor is an \textit{Sz}-system. That is the following

\begin{cor}\label{cor2.3}
If $\X^\prime=G\curvearrowright_T(X,\mathscr{X}^\prime,\mu)$ is a totally relatively weak-mixing extension of an (resp.~a weak) \textit{Sz}-system $\Y=G\curvearrowright_S(Y,\mathscr{Y},\nu)$, then $\X^\prime$ is an (resp.~a weak) \textit{Sz}-system.
\end{cor}

\begin{proof}
This follows immediately from Def.~\ref{def1.1}, Lemma~\ref{lem1.2} and Lemma~\ref{lem2.2}. Indeed, given any $A\in\mathscr{X}^\prime$ with $\mu(A)>0$ and any distinct elements $g_1,\dotsc,g_l\in G$, over any asymptotic F{\o}lner sequence $\{F_n\}_1^\infty$ in $(R,+)$ we can take $\delta>0,\varepsilon>0$ such that
\begin{gather*}
\left\{t\in R\colon \int_YS_{g_1}^tE_\mu(1_A|\Y)\dotsm S_{g_l}^tE_\mu(1_A|\Y)d\nu>2\varepsilon\right\}
\end{gather*}
has lower density $\ge\delta$ over $\{F_n\}_1^\infty$. Then by Lemma~\ref{lem2.2}, it follows that
\begin{gather*}
\left\{t\in R\colon \int_XT_{g_1}^t1_A\dotsm T_{g_l}^t1_A d\mu>\varepsilon\right\}
\end{gather*}
has lower density $\ge\delta$ over $\{F_n\}_1^\infty$. If $g_1=g_2$ then $T_{g_1}^t1_AT_{g_2}^t1_A=T_{g_1}^t1_A$ and so the above argument is still valid for any $g_1,\dotsc,g_l\in G$. This proves Corollary~\ref{cor2.3}.
\end{proof}

Another important consequence of Lemma~\ref{lem2.2} is the following, which is a generalization of \cite[Lemma~7.9]{Fur}.

\begin{cor}\label{cor2.4}
Let $\X^\prime=G\curvearrowright_T(X,\mathscr{X}^\prime,\mu)$ be a totally relatively weak-mixing extension of the $G$-system $\Y=G\curvearrowright_S(Y,\mathscr{Y},\nu)$ and let $g_1,\dotsc,g_l\in G$ with $g_i\not=g_j$ for $1\le i\not=j\le l$. Let
$\psi_1,\dotsc,\psi_l\in\mathfrak{L}^\infty(X,\mathscr{X}^\prime,\mu)$. Then for any $\delta>0$ and $\varepsilon>0$,
\begin{gather*}
\left\{t\in R\,\big{|}\nu\left\{y\in Y\colon\left|\int_X{\prod}_{i=1}^lT_{g_i}^t\psi_i\,d\mu_y-{\prod}_{i=1}^lE_\mu(\psi_i|\Y)(S_{g_i}^ty)\right|>\varepsilon\right\}>\delta\right\}
\end{gather*}
is of density $0$ over any weak F{\o}lner sequence $\{F_n\}_1^\infty$ in $(R,+)$.
\end{cor}

\section{Primitive extensions}\label{sec3}
Let $(G,\pmb{+})$ be an lcsc $R$-module if no an explicit declaration. As before, write $T_{tg}=T_g^t=g^t$ and $T_{g\pmb{+}h}=g\circ h$ for any $t\in R$ and $g,h\in G$. Let
$$
\X=G\curvearrowright_T(X,\mathscr{X},\mu)\xrightarrow[]{\textit{Id}_X}\X^\prime=G\curvearrowright_T(X,\mathscr{X}^\prime,\mu)\xrightarrow[]{\pi}\Y=G\curvearrowright_S(Y,\mathscr{Y},\nu)
$$
be any short factors series, where $(X,\mathscr{X},\mu)$ is a standard Borel probability space. Let
$$\mu=\int_Y\mu_yd\nu(y)$$
be the standard disintegration of $\mu$, over $\pi\colon\X\rightarrow\Y$, and simply write $\|\cdot\|_{2,y}=\|\cdot\|_{2,\mu_y}$ for $\nu$-\textit{a.e. }$y\in Y$.
As usual we shall say $\pi\colon\X^\prime\rightarrow\Y$ is \textbf{\textit{primitive}} if there exists direct sum
\begin{gather*}
G=G_{rc}\oplus G_{rw}
\end{gather*}
of nontrivial $R$-submodules of $G$ so that $\pi\colon\X^\prime\rightarrow\Y$ is relatively compact for $G_{rc}$ and totally relatively weak-mixing for $G_{rw}$.  Particularly, if $G_{rc}=G$, we say $\pi\colon\X^\prime\rightarrow\Y$ is a \textbf{\textit{relatively compact}} extension of $\Y$.
See \cite{FK,Fur} and \cite[$\S4$]{Dai-pre}.

We need a finitary version of the van der Waerden theorem for any $R$-module as follows.

\begin{Lem}[\cite{Dai-15}]\label{lem3.1}
Let $G_1\subset G_2\subset\dotsm\subset G_n\subset\dotsm$ be a sequence subsets of an $R$-semimodule $(G,\pmb{+})$ with $G=\bigcup_nG_n$; and let $F$ be a finite set of $G$ and $q\in \mathbb{N}$. Then there exists a number $N=N(q,F)$ such that whenever $n\ge N$ and $G_n=B_1\cup\dotsm\cup B_q$ is a partition of $G$ into $q$ sets, one of these $B_j$ contains a homothetic copy of $F$, $\{g_0^{}\pmb{+}rf\,|\,f\in F\}$, where $g_0^{}\in G$ and $r\in R$ with $r\not=0$.
\end{Lem}

Given any $g\in G$ and $m\in\mathbb{N}$, set $mg=(\stackrel{m\textit{-times}}{\overbrace{1+\dotsm+1}})g$, where $1$ is the identity element of $(R,+,\cdot)$. For our convenience we now introduce a notation.

\begin{defn}
We shall say an $R$-module $G$ is \textit{irreducible} relative to a subset $F$ of $G$ provided that if $mf\not=\e$ for any nonzero $f\in F$ and any $m\in\mathbb{N}$.
\end{defn}

Lemma~\ref{lem3.1} leads to the following useful result, which generalizes \cite[Lemma~7.11]{Fur}.

\begin{cor}\label{cor3.3}
Let $K\in\mathbb{N}$ and $F=\{g_1,g_2,\dotsc,g_l\}\subset G$ be any given, where $G$ is an irreducible $R$-module relative to $F$. Then there is a finite subset $Q\subset G$ and an integer $M\ge1$, such that for any coloring map
\begin{gather*}
k\colon G\rightarrow\{1,2,\dotsc,K\}
\end{gather*}
there exists some $g^\prime\in Q$ and some $m\in\mathbb{N}, 1\le m\le M$, such that
\begin{gather*}
k(g^\prime\pmb{+}mg_1)=k(g^\prime\pmb{+}mg_2)=\dotsm=k(g^\prime\pmb{+}mg_l).
\end{gather*}
\end{cor}

\begin{proof}
First we can choose a countable dense subgroup $G^\prime$ of $G$ with $F\subset G^\prime$. We will regard it as a $\mathbb{Z}_+$-semimodule.

According to Lemma~\ref{lem3.1} for $F=\{g_1,\dotsc,g_l\}\subset G^\prime$ and $q=K$, there exists a finite subset $G_L^\prime$ of $G^\prime$ such that for any coloring map $k\colon G_L^\prime\rightarrow\{1,2,\dotsc,K\}$, $k$ will be constant on a subset of $G_L^\prime$ of the form $\{g_0^{}\pmb{+}mg_1, g_0^{}\pmb{+}mg_2,\dotsc,g_0^{}\pmb{+}mg_l\}$ where $m\not=0$. But there are only finitely many possibilities for $m\in\mathbb{Z}_+$ and $g_0^{}\in G^\prime$ with $g_0^{}\pmb{+}mg_1\in G_L^\prime$ and $g_0^{}\pmb{+}mg_2\in G_L^\prime$, assuming, as we may, that $g_1\not=g_2$. Carrying this over to $G^\prime$ and then to $G$, we can conclude the statement.
\end{proof}

The implication of the relatively compactness of $\pi\colon G_{rc}\curvearrowright_T(X,\mathscr{X}^\prime,\mu)\rightarrow G_{rc}\curvearrowright_S(Y,\mathscr{Y},\nu)$ is summarized in the following lemma, which generalizes \cite[Lemma~7.10]{Fur}.

\begin{Lem}\label{lem3.4}
Let $\pi\colon G_{rc}\curvearrowright_T(X,\mathscr{X}^\prime,\mu)\rightarrow G_{rc}\curvearrowright_S(Y,\mathscr{Y},\nu)$ be relatively compact. Let $A\in\mathscr{X}^\prime$ with $\mu(A)>0$. Then we can find $A^\prime\subseteq A$ with $\mu(A^\prime)$ as close as we like to $\mu(A)$ having the property: For any $\varepsilon>0$ there exists a finite set of functions $\phi_1^{},\dotsc, \phi_K^{}$ in $\mathfrak{L}^2(X,\mathscr{X},\mu)$ and a function
\begin{gather*}
k\colon Y\times G_{rc}\rightarrow\{1,2,\dotsc,K\}
\end{gather*}
such that for every $h\in G_{rc}$, $\|T_h1_{A^\prime}-\phi_{k(y,h)}^{}\|_{2,y}$ for $\nu$-a.e. $y\in Y$.
\end{Lem}

\begin{proof}
This result follows easily from the characterization theorem of relatively compact extensions~\cite[Theorem~2.4]{Dai-pre}.
\end{proof}

Now we are ready to prove the following result which is one of our ladders for proving the multiple recurrence theorem.

\begin{prop}\label{prop3.5}
Let $(G,\pmb{+})$ be an irreducible lcsc $R$-module. If $\pi\colon\X^\prime\rightarrow\Y$
is a primitive extension of an \textit{Sz}-system $\Y$ over asymptotic F{\o}lner sequences in $(R,+)$, then $\X^\prime$ is also an \textit{Sz}-system over asymptotic F{\o}lner sequences in $(R,+)$.
\end{prop}

\begin{proof}
Our proof will follow from the framework of arguing of Furstenberg~\cite[Proposition~7.12]{Fur}. In what follows, let $A\in\mathscr{X}^\prime$ with $\mu(A)>0$, and let $g_1,\dotsc,g_l\in G$ be any given.

In view of Lemma~\ref{lem3.4}, we may assume that $1_A$ is \textit{FK a.p. for $G_{rc}$} described in Lemma~\ref{lem3.4}. Writing $\mu(A)=\int_Y\mu_y(A)d\nu(y)$, we see that if we set $a=\frac{1}{2}\mu(A)$, there exists a set $B\in\mathscr{Y}$ with $\nu(B)>0$ such that $\mu_y(A)\ge a$ for every $y\in B$.
We express the given elements $g_1,\dotsc,g_l\in G$ as follows:
\begin{gather*}
g_1=h_1\pmb{+}f_1, g_2=h_2\pmb{+}f_2, \dotsc,\ g_l=h_l\pmb{+}f_l,\quad \textit{where }\ h_i\in G_{rc}\textit{ and }f_i\in G_{rw}\textit{ for }1\le i\le l.
\end{gather*}
We let $\{F_n\}_1^\infty$ be any given asymptotic F{\o}lner sequence in $(R,+)$.

Let $a_1<a^l$. We shall show that there exists a measurable set $P\subset R$ with positive lower density over $\{F_n\}_1^\infty$ and $\eta>0$ such that for each $t\in P$ there is a set $B_t\in\mathscr{Y}$ with $\nu(B_t)>\eta$ such that
\begin{gather}\label{eq3.1}
\mu_y\left(\bigcap_{i=1}^lT_{h_i}^{-t}T_{f_i}^{-t}A\right)>a_1\quad\forall y\in B_t.
\end{gather}
This will implies Proposition~\ref{prop3.5}, since this leads to
$$
\mu\left(\bigcap_{i=1}^lT_{g_i}^{-t}A\right)>a_1\eta\quad \forall t\in P.
$$
The set $B_t$ will be defined by two requirements. For $a_1<a_2<a^l$ we shall require
\begin{gather}\label{eq3.2}
\mu_y\left(\bigcap_{i=1}^lT_{f_i}^{-t}A\right)>a_2\quad\forall y\in B_t\textit{ with }t\in P.
\end{gather}
Secondly we prescribe $\varepsilon_1>0$ such that if  $\mu_y\left(T_{f_i}^{-t}T_{h_i}^{-t}A\vartriangle T_{f_i}^{-t}A\right)<\varepsilon_1$ for all $1\le i\le l$, then (\ref{eq3.2}) implies (\ref{eq3.1}). Then we require
\begin{gather}\label{eq3.3}
\mu_y\left(T_{f_i}^{-t}T_{h_i}^{-t}A\vartriangle T_{f_i}^{-t}A\right)<\varepsilon_1,\quad 1\le i\le l,
\end{gather}
when $t\in P$ and $y\in B_t$.

Suppose now that $P$ and $\{B_t; t\in P\}$ have been found so that (\ref{eq3.3}) is satisfied and, in addition,
\begin{gather}\label{eq3.4}
S_{f_i}^ty\in B,\quad \forall y\in B_t,\ 1\le i\le l.
\end{gather}
Apply Corollary~\ref{cor2.4} with $\psi_1=\dotsm=\psi_l=1_A, \varepsilon<a^l-a_2$ and $\delta<\frac{1}{2}\min_{t\in P}\nu(B_t)$. Then
$$
\mu_y\left(\bigcap_{i=1}^lT_{f_i}^{-t}A\right)=\int_X\prod_{i=1}^lT_{f_i}^tfd\mu_y>\prod_{i=1}^lS_{f_i}^tE_\mu(1_A|\Y)(y)-\varepsilon=\prod_{i=1}^l\mu_{S_{f_i}^ty}(A)-\varepsilon\ge a^l-\varepsilon>a_2
$$
for any $y\in B_t$ but for a set of $y$ of measure $<\frac{1}{2}\nu(B_t)$ and for $t\in P$ outside a set of density $0$. Modifying $P$ and $B_t$ accordingly, we will be left with a set---call it again $P$---with positive lower density, and for each $t\in P$ a set---call it again $B_t$---with $\inf_{t\in P}\nu(B_t)>0$ such that for these $t$ and $y$, (\ref{eq3.2}) and (\ref{eq3.3}) are valid. As we have seen, $(\ref{eq3.3})+(\ref{eq3.4})\Rightarrow(\ref{eq3.3})+(\ref{eq3.2})\Rightarrow(\ref{eq3.1})$. Thus the problem is reduced to finding $P$ and $\{B_t; t\in P\}$ such that (\ref{eq3.3}) and (\ref{eq3.4}) are satisfied.

Let $\varepsilon_2<\frac{1}{2}\sqrt{\varepsilon_1}$ be any given. Recall that $1_A$ is \textit{FK a.p. for $G_{rc}$}. By Lemma~\ref{lem3.4}, we can then find functions $\phi_1^{},\dotsc,\phi_K^{}\in\mathfrak{L}^2(X,\mathscr{X},\mu)$ and a measurable coloring map
$$k\colon Y\times G_{rc}\rightarrow\{1,2,\dotsc,K\}$$
such that
$$
\left\|T_h1_A-\phi_{k(y,h)}^{}\right\|_{2,y}<\varepsilon_2\quad \forall h\in G_{rc}\textit{ and }\nu\textit{-a.e. }y\in Y.
$$
We now define a family of maps
$$
\big{\{}k_r\colon Y\times G\rightarrow\{1,2,\dotsc,K\}\big{\}}_{r\in R}
$$
by
$$
k_r(y,h\pmb{+}f)=k(S_f^ry,rh)\quad \forall (h,f)\in G=G_{rc}\oplus G_{rw}.
$$
We then have that for any $r\in R$, for $\nu$-\textit{a.e.} $y\in Y$,
\begin{gather}\label{eq3.5}
\left\|S_f^rT_h^r1_A-S_f^r\phi_{k_r(y,h\pmb{+}f)}^{}\right\|_{2,y}=\left\|T_h^r1_A-\phi_{k(S_f^ry,rh)}^{}\right\|_{2,S_f^ry}<\varepsilon_2.
\end{gather}
Fix $r\in R$ and $y\in Y$ and apply Corollary~\ref{cor3.3} to the map $k_r(y,\centerdot)$ on $G=G_{rc}\oplus G_{rw}$ with the integer $K$ and finite set $F=\{f_1,\dotsc,f_l,g_1,\dotsc,g_l\}$. Independently of $r$ and $y$ there is a finite subset $Q\subset G$ and an integer $M\ge1$ such that $k_r(y, g^\prime\pmb{+} mg_i)$ and $k_r(y,g^\prime\pmb{+}mf_i)$ both take on the same value for $1\le i\le l$ for some $g^\prime\in Q$ and some $m$ with $1\le m\le M$. If $\ell$ is this value, we write $\phi_{(r,y)}^{}=\phi_\ell^{}$. Then if $g^\prime=h^\prime\pmb{+}f^\prime$ and simply write $T_g^t=g^t$ and $S_g^t=g^t$, then
\begin{equation}\label{eq3.6}
\begin{split}
\left\|f_i^{mr}h_i^{mr}1_A-f_i^{mr}(h^{\prime -r}\phi_{(r,y)}^{})\right\|_{2,g^{\prime r}y}&=\left\|g^{\prime r}f_i^{mr}h_i^{mr}1_A-f^{\prime r}f_i^{mr}\phi_{(r,y)}^{}\right\|_{2,y}\\
&=\left\|f^{\prime r}f_i^{mr}h^{\prime r}h_i^{mr}1_A-f^{\prime r}f_i^{mr}\phi_{(r,y)}^{}\right\|_{2,y}\\
&<\varepsilon_2
\end{split}
\end{equation}
and
\begin{equation}\label{eq3.7}
\begin{split}
\left\|f_i^{mr}1_A-f_i^{mr}h^{\prime -r}\phi_{(r,y)}^{}\right\|_{2,g^{\prime r}y}&=\left\|g^{\prime r}f_i^{mr}1_A-f^{\prime r}f_i^{mr}\phi_{(r,y)}^{}\right\|_{2,y}\\
&=\left\|f^{\prime r}f_i^{mr}h^{\prime r}1_A-f^{\prime r}f_i^{mr}\phi_{(r,y)}^{}\right\|_{2,y}\\
&<\varepsilon_2
\end{split}
\end{equation}
for all $1\le i\le l$ by (\ref{eq3.5}), and moreover
\begin{equation}\label{eq3.8}
\begin{split}
\left\|1_A-h^{\prime -r}\phi_{(r,y)}^{}\right\|_{2,{g^\prime}^ry}&=\left\|g^{\prime r}1_A-f^{\prime r}\phi_{(r,y)}^{}\right\|_{2,y}\\
&=\left\|h^{\prime r}1_A-\phi_{k(f^{\prime r}y,rh^{\prime})}^{}\right\|_{2,y}\\
&<\varepsilon_2.
\end{split}
\end{equation}
What we have shown is that for every $r\in R$ and $\nu$\textit{-a.e.} $y\in Y$, there exist $m$ and $g^\prime$, both having a finite range of possibilities, such that
\begin{gather}\label{eq3.9}
\left\|f_i^{mr}h_i^{mr}1_A-f_i^{mr}1_A\right\|_{2,g^{\prime r}y}<\sqrt{\varepsilon_1}\quad (1\le i\le l).
\end{gather}
We are now ready to produce the set $P$ and the associated sets $B_t, t\in P$ such that both (\ref{eq3.3}) and (\ref{eq3.4}) hold for $(y,t), y\in B_t$.

For each $r\in R$ we form the set
\begin{gather}\label{eq3.10}
C_r=\bigcap_{i,m,g^\prime}\left(f_i^m\circ g^\prime\right)^{-r}B\in\mathscr{Y}
\end{gather}
where the intersection is taken over $i,m,g^\prime$ with $1\le i\le l, 1\le m\le M$ and $g^\prime\in Q$. Here we use the fact that $\Y=G\curvearrowright_S(Y,\mathscr{Y},\nu)$ is an \textit{Sz}-system. It follows from Lemma~\ref{lem1.2} that there exists a subset $P^\prime$ of $R$ such that
$P^\prime$ is of positive lower density over $\{F_n\}_1^\infty$ and $\nu(C_r)>\eta^\prime>0$ for each $r\in P^\prime$. Now let $y\in C_r$ for $r\in P^\prime$. There exists $m=m(r,y)$ and $g^\prime=g^\prime(r,y)$ such that (\ref{eq3.9}) holds and $f_i^{mr}(g^{\prime r}y)\in B$ for $1\le i\le l$. Let $J$ be the total number of possibilities for $(m,g^\prime)$. Then since $y\mapsto\mu_y$ is measurable, hence for a $\mathscr{Y}$-set $D_r\subset C_r$ with $\nu(D_r)>\frac{1}{J}\eta^\prime$, $m(r,y)$ and $g^\prime(r,y)$ take on a constant value, say $m(r), g(r)^\prime$. We now define $t(r)=m(r)r$, and set $P=\{t(r); r\in P^\prime\}$ and
\begin{gather}\label{eq3.11}
B_{t(r)}=g(r)^{\prime r}D_r\ \left(=S_{g(r)^\prime}^rD_r\right)\quad \forall r\in P^\prime.
\end{gather}
Then for any $t=t(r)\in P$, we have
\begin{gather}\label{eq3.12}
\nu(B_{t})>\frac{1}{J}\eta^\prime,\quad f_i^{t}B_{t}\subset B,\quad \left\|f_i^{t}h_i^{t}1_A-f_i^{t}1_A\right\|_{2,y}<\varepsilon_2 \quad(1\le i\le l)
\end{gather}
for any $y\in B_t$.

Finally we need to show that $P$ is of positive lower density over $\{F_n\}_1^\infty$. Indeed, let $\{F_{n}^\prime\}_1^\infty$ be any subsequence of $\{F_n\}_1^\infty$.
Set
$$
P_1^\prime=\{r\in P^\prime\,|\,m(r)=1\},\ \dotsc,\ P_M^\prime=\{r\in P^\prime\,|\,m(r)=M\}.
$$
Then by
\begin{equation*}
\begin{split}
\liminf_{n\to+\infty}\frac{|P^\prime\cap F_n^\prime|}{|F_n^\prime|}&=\liminf_{n\to+\infty}\frac{|P_1^\prime\cap F_n^\prime|+\dotsm+|P_M^\prime\cap F_n^\prime|}{|F_n^\prime|}\\
&=\liminf_{n\to+\infty}\frac{|(1P_1^\prime)\cap(1F_n^\prime)|+\dotsm+|(MP_M^\prime)\cap(MF_n^\prime)|}{|F_n^\prime|}\\
&\le\sum_{m=1}^M\limsup_{n\to+\infty}\frac{|P\cap(mF_n^\prime)|}{|F_n^\prime|}\\
&\le\sum_{m=1}^M\limsup_{n\to+\infty}\frac{c_m|P\cap F_{mn}^\prime|}{|F_{mn}^\prime|}\\
&\le(c_1+\dotsm+c_M)\limsup_{n\to+\infty}\frac{|P\cap F_n^\prime|}{|F_n^\prime|}
\end{split}\end{equation*}
we see that
$$
\limsup_{n\to+\infty}\frac{|P\cap F_n^\prime|}{|F_n^\prime|}>0.
$$
This implies that
$$
\liminf_{n\to+\infty}\frac{|P\cap F_n|}{|F_n|}>0.
$$
Now we have fulfilled all of our requirements, and thus this concludes the proof of Proposition~\ref{prop3.5}.
\end{proof}

The following result is contained in the proof of Proposition~\ref{prop3.5}.

\begin{prop}\label{prop3.6}
Let $(G,\pmb{+})$ be an lcsc $R$-module. If $\pi\colon\X^\prime\rightarrow\Y$
is a primitive extension of the system $\Y$ which satisfies that for any $B\in\mathscr{Y}$ with $\nu(B)>0$ and any $g_1,\dotsc,g_l\in G$ there exists some integer $N=N(B,g_1,\dotsc,g_l)\ge1$ such that for any weak F{\o}lner sequence $\{F_n\}_1^\infty$ in $(R,+)$
\begin{gather}\label{eq3.13}
\limsup_{n\to+\infty}\frac{1}{|F_n|}\int_{F_n}\nu\left(B\cap S_{g_1}^{-m^\prime t}B\cap\dotsm\cap S_{g_l}^{-m^\prime t}B\right)dt>0
\end{gather}
for some integer $m^\prime$ with $1\le m^\prime\le N$,
then $\X^\prime$ also possesses this property.
\end{prop}

\begin{proof}
We first note that there is no loss of generality in assuming that $G$ is irreducible relative to the finite set $F=\{g_1,\dotsc,g_l\}\subset G$. Let $A\in\mathscr{X}$ with $\mu(A)>0$ and then $B\in\mathscr{Y}$ with $\nu(B)>0$ be as in the proof of Proposition~\ref{prop3.5}.
Similar to (\ref{eq3.10}), for each $t\in R$ we form the set
\begin{equation*}
C_t=\bigcap_{i,m,g^\prime}\left(f_i^m\circ g^\prime\right)^{-m^\prime t}B\in\mathscr{Y} \eqno{(\ref{eq3.10})^\prime}
\end{equation*}
for some integer $m^\prime\ge1$,
where the intersection is taken over $i,m,g^\prime$ with $1\le i\le l, 1\le m\le M$ and $g^\prime\in Q$. Here we use the fact that $\Y=G\curvearrowright_S(Y,\mathscr{Y},\nu)$ satisfies (\ref{eq3.13}) with the set
$$F_B:=\{m^\prime mf_i\pmb{+}m^\prime g^\prime\,|\,1\le i\le l, 1\le m\le M, g^\prime\in Q\}$$
in place of $F$.
Then for any weak F{\o}lner sequence $\{F_n\}_1^\infty$ in $(R,+)$ there exists an integer $m^\prime$ with $1\le m^\prime\le N(B,F_B)$ and a subset $P$ of $R$ such that
$P$ is of positive upper density  and $\nu(C_t)>\eta^\prime>0$ for each $t\in P$, where $\eta^\prime$ relies on $\{F_n\}_1^\infty$.

Now let $y\in C_t$ for $t\in P$.
There exists $m=m(t,y)$ and $g^\prime=g^\prime(t,y)$ such that
\begin{equation*}
\left\|f_i^{m^\prime mt}h_i^{m^\prime mt}1_A-f_i^{m^\prime mt}1_A\right\|_{2,g^{\prime m^\prime t}y}<\sqrt{\varepsilon_1}\ \textit{ and }\ f_i^{m^\prime mt}(g^{\prime m^\prime t}y)\in B\quad (1\le i\le l).\eqno{(\ref{eq3.9})^\prime}
\end{equation*}
for all $1\le i\le l$. Let $J$ be the total number of possibilities for $(m,g^\prime)\in\{1,2,\dotsc,M\}\times Q$. Hence for a $\mathscr{Y}$-set $D_t\subset C_t$ with $\nu(D_t)>\frac{1}{J}\eta^\prime$, $m(t,y)$ and $g^\prime(t,y)$ take on a constant value, say $m_t, g_t^\prime$. We now define
\begin{equation*}
B_{t}={g_t^\prime}^{m^\prime t}D_t\ \left(=S_{g_t^\prime}^{m^\prime t}D_t\right)\quad \forall t\in P.\eqno{(\ref{eq3.11})^\prime}
\end{equation*}
Then for any $t\in P$, we have
$$
\nu(B_{t})>\frac{1}{J}\eta^\prime,\quad f_i^{m_tm^\prime t}B_{t}\subset B,\quad \left\|f_i^{m_tm^\prime t}h_i^{m_tm^\prime t}1_A-f_i^{m_tm^\prime t}1_A\right\|_{2,y}<\varepsilon_2 \quad(1\le i\le l) \eqno{(\ref{eq3.12})^\prime}
$$
for any $y\in B_t$.
Set
$$
P_1=\{t\in P\,|\,m_t=1\},\ \dotsc,\ P_M=\{t\in P\,|\,m_t=M\}.
$$
Then there is some integer $k$ with $1\le k\le M$ such that
$$
\limsup_{n\to+\infty}\frac{|P_k\cap F_n|}{|F_n|}>0.
$$
This implies that
\begin{gather*}
\limsup_{n\to+\infty}\frac{1}{|F_n|}\int_{F_n}\mu\left(A\cap {g_1}^{-km^\prime t}A\cap\dotsm\cap {g_l}^{-km^\prime t}A\right)dt>0.
\end{gather*}
Now noting that $1\le km^\prime\le M\cdot N(B,F_B)$ and that $M\cdot N(B,F_B)$ is independent of $\{F_n\}_1^\infty$, this thus concludes the proof of Proposition~\ref{prop3.6}.
\end{proof}

\section{The measure-theoretic multiple recurrence}\label{sec4}
This section will be devoted to proving the multiple recurrence theorems, namely Theorem~\ref{thm0.1} and \ref{thm0.3}, stated in Introduction.

\subsection{Proof of Theorem~\ref{thm0.1}}
Let $\pi\colon\X=(X,\mathscr{X},\mu,G)\rightarrow\Y=(Y,\mathscr{Y},\nu,G)$ be a standard Borel extension of a factor $\Y$, where $G$ is an irreducible lcscN module over a syndetic ring $(R,+,\cdot)$. Then the following proposition is a slightly general statement than Theorem~\ref{thm0.1}.

\begin{prop}\label{prop4.1}
If $\Y$ is an \textit{Sz}-factor of $\X$, then $\X$ is an \textit{Sz}-system, over any asymptotic F{\o}lner sequence in $(R,+)$.
\end{prop}

\begin{proof}
Let
\begin{equation*}
\X\rightarrow\X_\eta\rightarrow\dotsm\rightarrow\X_{\xi+1}\xrightarrow[]{\pi_{\xi+1,\xi}}\X_\xi\rightarrow\dotsm\rightarrow\X_2\xrightarrow[]{\pi_{2,1}}\X_1\xrightarrow[]{\pi_{1,0}}\Y
\end{equation*}
be the Furstenberg factors chain of $\X=G\curvearrowright_T(X,\mathscr{X},\mu)$ starting from $\Y$. Since $\Y$ is an \textit{Sz}-system over any asymptotic F{\o}lner sequence in $(R,+)$, hence from Corollaries~\ref{cor1.4}, \ref{cor2.3} and Proposition~\ref{prop3.5} we can see that $\X$ is an \textit{Sz}-system by transfinite induction.
This thus concludes the proof of Proposition~\ref{prop4.1}.
\end{proof}

\subsection{Proof of Theorem~\ref{thm0.3}}
Based on Furstenberg Structure Theorem, by transfinite induction, Theorem~\ref{thm0.3} follows immediately from Lemma~\ref{lem1.5}, Corollary~\ref{cor2.3} and Proposition~\ref{prop3.6}.

\section{Multiple Birkhoff center and pointwise multirecurrence}\label{sec5}
Let $X$ be a compact metric space $X$ in the sequel. In this section, the topological dynamical system (for brevity, \textit{t.d.s.})
$T\colon G\times X\rightarrow X$ or $G\curvearrowright_TX$ we now consider is as in $\S0.3$, where $G$ is an lcscN module over a syndetic ring $(R,+,\cdot)$ with a Haar measure $|\cdot|$ (or write $dt$). In addition, assume $R$ is not compact, i.e., $|R|=\infty$.
\subsection{Poisson stable motions}
Given any sequence $(T_n)_{1}^\infty$ of elements of $R$, for convenience, we shall say that $T_n\to+\infty$ as $n\to+\infty$ if for any compacta $K\subset$ of $R$ there exists a positive integer $L=L(K)$ for which
\begin{gather*}
n\ge L\Rightarrow T_n\not\in K.
\end{gather*}
It will be a useful fact that if $T_n\to+\infty$ in $R$ then $t+T_n\to+\infty$ for every $t\in R$. Since there does not need to have any order in $R$, the above notion ``$T_n\to+\infty$'' ought to be of interest.

According to \cite[Def.~4.15]{Pat}, a sequence $\{R_n\}_{1}^\infty$ of nonnull compact subsets of $R$ is called a \textit{summing sequence} in $(R,+)$ if the following conditions are satisfied:
\begin{enumerate}
\item[(1)] $R=\bigcup_{1}^{\infty}R_n$;
\item[(2)] $R_n\subset\textrm{Int}(R_{n+1})\ \forall n\ge1$, where $\textrm{Int}(\centerdot)$ denotes the interior of a set;
\item[(3)] $|(r+R_n)\vartriangle R_n|\cdot|R_n|^{-1}\to0$ as $n\to+\infty$, for any $r\in R$ and then uniformly for $r$ in any compacta of $R$.
\end{enumerate}
Since $(R,+)$ is an abelian, lcsc, and Hausdorff group here, it is amenable and $\sigma$-compact. Therefore, there always exists a summing sequence in $(R,+)$; see e.g. \cite[Theorem~4.16]{Pat}. Clearly a summing sequence is a weak F{\o}lner sequence by condition (3).
From now on,
\begin{itemize}
\item let $\mathcal{R}=\{R_n\}_{1}^\infty$
be an arbitrarily fixed summing sequence in $(R,+)$.
\end{itemize}

For a sequence $(T_n)_{n=1}^\infty$ of elements of $R$, we say $T_n\xrightarrow[]{\mathcal{R}}+\infty$ as $n\to+\infty$ if for any integer $n\ge1$ there is some $L=L(n)>1$ such that $T_\ell\not\in R_n$ as $\ell\ge L$.
Then the following fact is obvious.

\begin{Lem}\label{lem5.1}
$T_n\to+\infty$ in $R$ if and only if $T_n\xrightarrow[]{\mathcal{R}}+\infty$. Thus if $T_n\xrightarrow[]{\mathcal{R}}+\infty$ then $t+T_n\xrightarrow[]{\mathcal{R}}+\infty$ for every $t\in R$.
\end{Lem}

\begin{proof}
This follows from the fact that every compact subset $K$ of $R$ must be contained by $R_n$ as $n$ sufficiently big.
\end{proof}

Next, following Furstenberg's idea~\cite[Chap.~2]{Fur} we now introduce a basic concept\,--\,multiple recurrent motion for $G$-action system as follows:

\begin{defn}\label{def5.2}
For $G\curvearrowright_TX$ a point $p\in X$ is said to be
\textit{multiply recurrent} (or \textit{multiply Poisson stable}) if for any $l\ge2$ and any sample elements $g_1, \dotsc, g_l\in G$, one can find a sequence of times $t_n\to+\infty$ in $R$ so that
$g_i^{t_n}p\to p$ as $n\to+\infty$, simultaneously for $i=1,\dotsc,l$.
\end{defn}

\subsection{Multiple recurrence}\label{sec5.2}
Although a multiply recurrent point of $G\curvearrowright_TX$ does not need to lie in the multiple Birkhoff center $\Omega_\gamma(T)$ according to Def.~\ref{def0.10} and Def.~\ref{def5.2}, yet the structure of $\Omega_\gamma(T)$ is made clear by the multiple recurrence via the following two lemmas.

\begin{Lem}\label{lem5.3}
Given any $l$ elements $g_1, \dotsc, g_l$ in $G$, the $(g_1,\dotsc,g_l)$-multiple recurrent points of $G\curvearrowright_TX$ are everywhere dense in $\Omega_\gamma(T)$.
\end{Lem}

\begin{proof}
We consider the subsystem $G\curvearrowright_T\Omega_\gamma(T)$ instead of $G\curvearrowright_TX$. Let $p\in \Omega_\gamma(T)$ be any point and $\varepsilon>0$ be an arbitrary number, and let $g_1, \dotsc, g_l$ be any given $l$ elements in $G$ where $l\in \mathbb{\mathbb{N}}$. It is required to prove that in the relative $\varepsilon$-ball $U_0=B_\varepsilon(p)$ around $p$ in $\Omega_\gamma(T)$ there can be found a $(g_1,\dotsc,g_l)$-recurrent point for $G\curvearrowright_T\Omega_\gamma(T)$.

Because of the regional multiple recurrence of $T$ on $\Omega_\gamma(T)$, there can be found some time $\tau_1\in R, \tau_1\not\in R_1$, simply written as $\tau_1>R_1$, so that
\begin{gather*}
U_0\cap g_1^{-\tau_1}U_0\cap\dotsm\cap g_l^{-\tau_1}U_0\not=\varnothing.
\end{gather*}
Since the intersection of $l+1$ open sets is still an open set, there can be found a point and a number, say $p_1\in U_0$ and $\varepsilon_1>0$, such that
\begin{gather*}
B_{\varepsilon_1}(p_1)\subset U_0\cap g_1^{-\tau_1}U_0\cap\dotsm\cap g_l^{-\tau_1}U_0.
\end{gather*}
We simply write $U_1=B_{\varepsilon_1/2}(p_1)$. By virtue of the same regional multiple recurrence there can be found an element $\tau_2>R_2$ in $R$ with
\begin{gather*}
U_1\cap g_1^{-\tau_2}U_1\cap\dotsm\cap g_l^{-\tau_2}U_1\not=\varnothing
\end{gather*}
and there can be found a point $p_2\in U_1$ and a number $\varepsilon_2>0$ such that
\begin{gather*}
B_{\varepsilon_2}(p_2)\subset U_1\cap g_1^{-\tau_2}U_1\cap\dotsm\cap g_l^{-\tau_2}U_1.
\end{gather*}
Obviously, $\varepsilon_2\le\varepsilon_1/2$. We set $U_2=B_{\varepsilon_2/2}(p_2)$. Next there can be found a point $p_3$ and a number $\varepsilon_3>0$ such that
\begin{gather*}
B_{\varepsilon_3}(p_3)\subset U_2\cap g_1^{-\tau_3}U_2\cap\dotsm\cap g_l^{-\tau_3}U_2,
\end{gather*}
where $\tau_3\in R, \tau_3>R_3$ and $\varepsilon_3\le\varepsilon_2/2$.

Continuing this process without end and noting that $\overline{U}_n\subseteq U_{n-1}$ for $n=1,2,\dotsc$ and, besides, that $\textrm{diam}(\overline{U}_n)<\varepsilon_n\le\varepsilon/2^{n-1}$, we obtain because of the compactness of the multiple Birkhoff center $\Omega_\gamma(T)$ a point $q$ as the intersection of the sets $U_n$:
$\{q\}=\bigcap_{n=1}^\infty U_n$.
Since $g_i^{\tau_{n+1}}q\in U_n$ for $n=1,2,\dotsc$ and $i=1, \dotsc, l$,
we can see that
$g_i^{\tau_n}q\to q$ as $n\to+\infty$
simultaneously for $i=1,\dotsc,l$. Clearly $\tau_n\to+\infty$ by Lemma~\ref{lem5.1}.

This completes the proof of Lemma~\ref{lem5.3}.
\end{proof}

We note that this result is still valid if, instead of $\Omega_\gamma(T)$, any compact set be taken which possesses the property of regional multiple recurrence relative to $(R_n)_{1}^\infty$.

\begin{Lem}\label{lem5.4}
Given any $g_1, \dotsc, g_l\in G$, the $(g_1,\dotsc,g_l)$-multiple recurrent points of $G\curvearrowright_TX$ form a residual subset of $\Omega_\gamma(T)$.
\end{Lem}

\begin{proof}
Let $\{\varepsilon_n\}$ be a sequence of positive numbers with $\varepsilon_n\downarrow0$ as $n\to+\infty$. Define sets
$$
X_n=\left\{p\in X\,\big{|}\,\max_{1\le i\le l}d(p,g_i^tp)\ge\varepsilon_n\ \forall t>R_n\right\};
$$
where $X_n$ may be an empty set. Obviously, all the points $p\in X_n$ are not $(g_1,\dotsc,g_l)$-recurrent for $T$, and it is easy to check that every point which is not $(g_1,\dotsc,g_l)$-recurrent for $T$ lies in some $X_n$.
Clearly, $X_n$ is closed by the continuity of $T(tg,x)$ with respect to $(t,g,x)\in R\times G\times X$. Furthermore, $X_n$ is nowhere dense in $\Omega_\gamma(T)$ by Lemma~\ref{lem5.3}. This implies that $\Omega_\gamma(T)-\bigcup_{n=1}^\infty X_n$ is a residual subset of $\Omega_\gamma(T)$ and the proof of Lemma~\ref{lem5.4} is thus completed.
\end{proof}

The full probability property of multiply recurrent points of $G\curvearrowright_TX$ comes from the following theorem by using a measure-theoretic version of the proof of Lemma~\ref{lem5.3}.

\begin{Thm}[C.~Carath\'{e}odory~\cite{Car} for $G=\mathbb{R}$]\label{thm5.5}
Given any $l$ elements $g_1, \dotsc, g_l\in G$, the set $\mathcal{R}_{g_1,\dotsc,g_l}(T)$ of all the $(g_1,\dotsc,g_l)$-multiple recurrent points of $G\curvearrowright_TX$ is a full probability $G_\delta$ subset of $X$.
\end{Thm}

\begin{proof}
By the proof of Lemma~\ref{lem5.4}, it is easy to see that $\mathcal{R}_{g_1,\dotsc,g_l}(T)$ is a $G_\delta$-subset of $X$. Without loss of generality, let $\mu$ be an arbitrarily given ergodic Borel probability measure of $G\curvearrowright_TX$ such that $\textrm{supp}(\mu)$ does not consist of a fixed point of $T$; otherwise $\mu(\mathcal{R}_{g_1,\dotsc,g_l}(T))=1$. By contradiction, let $\mu(\mathcal{R}_{g_1,\dotsc,g_l}(T))<1$; then since $\mathcal{R}_{g_1,\dotsc,g_l}(T)$ is $T$-invariant, it follows at once that
\begin{gather*}Y=X-\mathcal{R}_{g_1,\dotsc,g_l}(T)\end{gather*}
is an $T$-invariant Borel set of $\mu$-measure $1$.
We now consider $G\curvearrowright_T(Y,\mu)$ instead of $G\curvearrowright_TX$. We shall prove that there is a $(g_1,\dotsc,g_l)$-recurrent point of $G\curvearrowright_TX$ that lies in $Y$. This is a contradiction.

Since $Y$ is a Borel subset of $X$ and $\mu$ is regular, by Theorem~\ref{thm0.3} it follows that
there can be found a closed subset $Y_0$ of $X$ with
$$
Y_0\subset Y,\quad 0<\mu(Y_0)\le\frac{1}{2}\quad \textrm{and} \quad \textrm{diam}(Y_0)\le 1
$$
and there is some element $\tau_1>R_1$ in $R$ (i.e. $\tau_1\in R, \tau_1\not\in R_1$) so that
$$
\mu\big{(}Y_0\cap g_1^{-\tau_1}Y_0\cap\dotsm\cap g_l^{-\tau_1}Y_0\big{)}>0.
$$
Then there can be found a closed subset $Y_1$ of $X$ with
$$
Y_1\subset Y_0\cap g_1^{-\tau_1}Y_0\cap\dotsm\cap g_l^{-\tau_1}Y_0,\quad 0<\mu(Y_1)\le\frac{1}{2^2}, \quad \textrm{diam}(Y_1)\le\frac{1}{2},
$$
Next by Theorem~\ref{thm0.3} again, there can be found an element $\tau_2>R_2$ in $R$ (i.e. $\tau_2\in R, \tau_2\not\in R_2$) with
$$
\mu\big{(}Y_1\cap g_1^{-\tau_2}Y_1\cap\dotsm\cap g_l^{-\tau_2}Y_1\big{)}>0.
$$
Further there can be found a closed subset $Y_2$ of $X$ with
$$
Y_2\subset Y_1\cap g_1^{-\tau_2}Y_1\cap\dotsm\cap g_l^{-\tau_2}Y_1,\quad 0<\mu(Y_2)\le\frac{1}{2^3}, \quad \textrm{diam}(Y_2)\le\frac{1}{2^2},
$$
and an element $\tau_3>R_3$ in $R$ with
$$
\mu\big{(}Y_2\cap g_1^{-\tau_3}Y_2\cap\dotsm\cap g_l^{-\tau_3}Y_2\big{)}>0.
$$
Continuing this process without end and noting that $Y_n\subseteq Y_{n-1}$ for $n=1,2,\dotsc$ and, besides, that $\textrm{diam}(Y_n)\le1/2^n$, we obtain because of the compactness of the space $X$ a point $q\in Y$ as the intersection of the sets $Y_n$:
$$
\{q\}=\bigcap_{n=1}^\infty Y_n.
$$
Since
$$
g_i^{\tau_{n}}q\in Y_{n-1}\quad \textrm{for }n=1,2,\dotsc\textrm{ and }i=1, \dotsc, l,
$$
we can see that $\tau_n\to+\infty$ and
$$
g_i^{\tau_n}q\to q\quad \textrm{as }n\to+\infty
$$
simultaneously for $i=1,\dotsc,l$. We thus arrive at a contradiction.

This completes the proof of Theorem~\ref{thm5.5}.
\end{proof}

Now combining Theorem~\ref{thm5.5} with Lemma~\ref{lem5.4} follows at once the following result:

\begin{prop}\label{prop5.6}
Given any $g_1,\dotsc,g_l\in G$, the $(g_1,\dotsc,g_l)$-multiple recurrent points of $G\curvearrowright_TX$ form a residual and full probability set in the multiple Birkhoff center $\Omega_\gamma(T)$.
\end{prop}

\begin{cor}\label{cor5.7}
If $G$ is countable, then the multiply recurrent points of $G\curvearrowright_TX$ form a $G_\delta$ set of full probability, which is residual in $\Omega_\gamma(T)$.
\end{cor}

\begin{proof}
Since $G$ is countable, hence $\bigcap_{l\in\mathbb{N}}\bigcap_{(g_1,\dotsc,g_l)\in G^l}\mathcal{R}_{g_1,\dotsc,g_l}(T)$ is $G_\delta$ in $X$ and residual in $\Omega_\gamma(T)$ and of full probability. This proves the corollary.
\end{proof}

Particularly under the discrete topology, $\mathbb{Z}^d$ is a countable Noetherian $\mathbb{Z}$-module.
Next we will consider a $\mathbb{Z}$-acting dynamical system. Let $T\colon X\rightarrow X$ be a homeomorphism of the compact metric space $X$, which has the multiple Birkhoff center $M_\gamma(T)$ and $([-n,\dotsc,n])_{1}^\infty$ is a summing sequence in $(\mathbb{Z},+)$.

In 1994~\cite{Gla94}, by using the topological ``ergodic decomposition'' theory developed in \cite{EGS, Vee} E.~Glasner showed that if $(X,T)$ is \textit{minimal and topologically weakly mixing}, then for any $l\ge1$ the $(1,2,\dotsc,l)$-recurrent points of $T$ form a dense $G_\delta$ subset in $X$. Whenever $(X,T)$ is \textit{minimal}, then it is also well known that there always exists a dense $G_\delta$ set of points which are $(1,2,\dotsc,l)$-recurrent for any $l\ge1$. This result appears scattered in many literature; see, e.g., \cite[Theorem~2.5]{HKM}. Moreover, in \cite[Theorem~3.12]{KLOY}, D.~Kwietniak et al. proved that the set of the $(1,2,\dotsc,l)$-recurrent points of $T$ is of full probability.

The $\mathbb{Z}$-time version of Corollary~\ref{cor5.7} can be stated as follows:

\begin{prop}\label{prop5.8}
Let $T\colon X\rightarrow X$ be a homeomorphism of the compact metric space $X$, which has the multiple Birkhoff center $M_\gamma(T)$. Then the set of all the multi-recurrent points of the single homeomorphism $T$ is residual and of full probability in $M_\gamma(T)$.
\end{prop}

Here for the multiple recurrence, our sample times $t_1,\dotsc, t_l\in\mathbb{Z}$ are not necessarily to be positive.

We note that since there exists a dynamical system on the $2$-dimensional torus $\mathbb{T}^2$ that is uniquely ergodic with the measure center $o=(0,0)$ and has the multiple Birkhoff center $\mathbb{T}^2$ (cf.~\cite[Example~6.16]{NS}), it is complementary between the topological structure and the full probability of the set of all the multi-recurrent points. Propositions~\ref{prop5.6} and \ref{prop5.8} show that for any topological dynamical system $G\curvearrowright_TX$, the multiply recurrent motions are by no means rare from both the topological viewpoint and the point of view of measure theory.
\section*{\textbf{Acknowledgments}}%
This work was partly supported by National Natural Science Foundation of China grant $\#$11271183 and PAPD of Jiangsu Higher Education Institutions.



\end{document}